\documentclass{amsart}
\usepackage[a4paper,centering,textwidth=145mm,textheight=195mm]{geometry}
\usepackage[all]{xypic}
\usepackage{BOONDOX-cal}

\newcommand{\invo}{\overline{\rule{2.5mm}{0mm}\rule{0mm}{4pt}}}
\newcommand{\calQ}{Q_\fq}
\newcommand{\calI}{\mathcal{I}}
\newcommand{\Z}{\mathbb{Z}}
\newcommand{\FQ}{F}
\newcommand{\fq}{\mathcal{F}}
\newcommand{\op}{\dag}

\DeclareMathOperator{\Sesq}{Sesq}
\DeclareMathOperator{\Quad}{Quad}
\DeclareMathOperator{\Herm}{Herm}
\DeclareMathOperator{\sw}{sw}

\DeclareMathOperator{\charac}{char}
\DeclareMathOperator{\End}{End}
\DeclareMathOperator{\Hom}{Hom}
\DeclareMathOperator{\Trd}{Trd}
\DeclareMathOperator{\Sym}{Sym}
\DeclareMathOperator{\Skew}{Skew}

\newtheorem{prop}{Proposition}[section]
\newtheorem{lemma}[prop]{Lemma}
\newtheorem*{thm}{Theorem}
\newtheorem*{question}{Question}
\newtheorem{corol}[prop]{Corollary}
\newtheorem*{corol*}{Corollary}

\theoremstyle{remark}
\newtheorem*{remk}{Remark}

\title[New proof of the Artin--Springer theorem]{A new proof of the
  Artin--Springer theorem\\ in Schur index $2$} 
\author{Anne Qu\'eguiner-Mathieu}
\address{Universit\'e Sorbonne Paris Nord\\
Institut Galilée \\
LAGA - CNRS (UMR 7539)\\
F-93430 Villetaneuse, France}
\email{queguin@math.univ-paris13.fr}

\author{Jean-Pierre Tignol}
\address{ICTEAM Institute, Box L4.05.01\\
UCLouvain\\
B-1348 Louvain-la-Neuve, Belgium}
\email{jean-pierre.tignol@uclouvain.be}

\thanks{The second author acknowledges support
    from the Fonds de la Recherche Scientifique--FNRS under grants
    n$^\circ$~J.0149.17 and J.0159.19. Both authors acknowledge
    support from Wallonie--Bruxelles International and the French
    government in the framework of ``Partenariats Hubert Curien''
    (Projects \emph{Groupes alg\'ebriques, anisotropie et invariants
      cohomologiques} and \emph{$2$-alg\`ebres de degr\'e~$8$ et
      d'exposant~$2$}.)}

\date{Typeset \today}

\begin{document}
\maketitle

\begin{abstract}
We provide a new proof of the analogue of the Artin--Springer theorem
for groups of type~$\mathsf{D}$ that can be represented by similitudes over an
algebra of Schur index $2$: an anisotropic generalized quadratic form
over a quaternion algebra $Q$ remains anisotropic after generic
splitting of $Q$, hence also under odd-degree field extensions of the
base field. Our proof is characteristic-free and does not use the
excellence property.  
\end{abstract}

A well-known theorem, published by Springer in 1952, but already known
to Artin\footnote{See Witt's review of Springer's paper on
  Zentralblatt.} in 1937, states that an anisotropic quadratic 
form over a field remains anisotropic after odd-degree field
extensions (see~\cite{Sp} for Springer's original paper
and~\cite[Cor.~18.5]{EKM} for a characteristic-free
proof). Equivalently, this means that anisotropy is preserved under odd degree field extensions for the group of
similitudes of a quadratic form over a field. 
Whether the same
property holds for every simple linear algebraic group of
type~$\mathsf{D}$ 
is a largely open question\footnote{A weaker question is whether
  non-hyperbolic hermitian forms remain non-hyperbolic under
  odd-degree field extensions. An affirmative answer was given by
  Bayer--Lenstra~\cite{BL}, and provides a partial solution to a
  question raised by Serre~\cite[5.3(ii)]{Se} in the paper where he
  formulates his Conjectures I and II: is it possible to associate to
  every type of 
  semisimple group $G$ an integer $d$ such that the scalar extension
  map $H^1(k,G)\to H^1(K,G)$ is injective for every field extension
  $K/k$ of degree prime to $d$? }, stated for instance
in~\cite[\S\,7]{ABGV}. No counterexample is known; see~\cite{BQM} for
a survey of known results.

A general strategy to tackle this question was proposed, and
implemented in a special case, by
Parimala--Sridharan--Suresh~\cite{PSS}. 
Assume an algebraic group $G$ over a field $k$ can be represented after scalar extension to a field $F$ as the group of similitudes of an anisotropic quadratic form. For every odd-degree field extension $\ell/k$, the group remains anisotropic after scalar extension to a composite extension $F\cdot \ell$ of odd degree over $F$, by the Artin--Springer theorem. Therefore, it is anisotropic over $\ell$. 

In characteristic different from~$2$, Karpenko~\cite{Kar} proved the
converse: if an anisotropic linear algebraic group of
type~$\mathsf{D}$ remains anisotropic over every odd-degree extension
of the base field, then there exists an extension over which it can be
represented (up to isogeny) as the group of similitudes of an
anisotropic quadratic form. The same holds in characteristic~$2$, as
was shown by Medhi~\cite{Medhi}. Therefore, whether the
Artin--Springer property holds for groups of type~$\mathsf{D}$ is
actually equivalent to the following question:

\begin{question}
  Let $G$ be an anisotropic simple linear algebraic group of
  type~$\mathsf{D}$ over a field $k$. Is there a field extension $F$
  of $k$ such that the group $G_F$ obtained from $G$ by scalar
  extension to $F$ can be represented (up to isogeny) as the group of
  similitudes of an anisotropic quadratic form over $F$?
\end{question}

Parimala--Sridharan--Suresh~\cite{PSS} dealt with the case where $G$
can be represented as the group of similitudes of a skew-hermitian
space over a quaternion division algebra in characteristic different
from~$2$. The question above then amounts to the following: do
anisotropic skew-hermitian spaces remain anisotropic after scalar
extension to a field that splits the quaternion algebra? It has an
affirmative answer with $F$ 
the function field of the Severi--Brauer conic of $2$-dimensional
right ideals in the quaternion algebra. 
Note that the same argument as in~\cite{PSS} also works in
characteristic $2$, combining results from~\cite{BD} and \cite{MT} as
explained below. 

To address those questions in characteristic~$2$, a choice needs to be made between two different
ways of representing simple linear algebraic groups of
type~$\mathsf{D}$: either by similitudes of a generalized quadratic
form, as in~\cite[\S\,3]{Tits}, or by similitudes of a quadratic pair,
as in~\cite[\S\,26]{BoI}. In this paper, we take the option of representing groups of type~$\mathsf{D}$ by
generalized quadratic spaces. We give a new and characteristic-free
proof of the Parimala--Sridharan--Suresh result in the following form:

\begin{thm}
  Let $Q$ be a quaternion division algebra over an arbitrary field
  $k$. Every anisotropic generalized quadratic space over $Q$ remains
  anisotropic after scalar extension to the function field $\FQ$ of
  the Severi--Brauer conic of $Q$, hence also after scalar extension to any odd degree extension of $k$. 
\end{thm}

To underline the contrast with our argument and the broader scope of the
proof in~\cite{PSS}, we sketch the two-step approach of the latter. First,
\cite[Cor.~2.2]{PSS} shows that the function field of a conic
satisfies the following ``excellence'' property: for every hermitian
or skew-hermitian form $h$ over a central simple algebra with
involution over a field $k$, the anisotropic kernel of the extension
$h_\FQ$ of $h$ to the function field $\FQ$ of a smooth conic over $k$
is extended from~$k$. Second, \cite[Prop.~3.3]{PSS} shows that if $Q$
is a quaternion division algebra and $\FQ$ is the function field of
its Severi--Brauer conic, then the
skew-hermitian forms $h$ over $Q$ such that $h_\FQ$ is hyperbolic are
already hyperbolic over $Q$. 
Consider an anisotropic skew-hermitian form $h$
over $Q$. 
The excellence property yields a
skew-hermitian form $h'$ over $Q$ whose scalar extension $h'_\FQ$ to
$\FQ$ is the anisotropic kernel of $h_\FQ$. Then $h\perp-h'$ is
hyperbolic after scalar extension to $\FQ$, hence also over
$Q$. Therefore, $h$ is isometric to $h'$ by Witt cancellation, and
$h_\FQ$ is anisotropic. Each of the steps has been 
established in characteristic~$2$: the excellence property of
function fields of smooth conics is established in~\cite{MT}, and 
the analogue of~\cite[Prop.~3.3]{PSS}
is proved by Becher--Dolphin in~\cite{BD} in 
the equivalent language of quadratic pairs.

Our approach produces a direct and characteristic--free proof of the theorem above, which does not use the excellence property, see~\S\ref{sec:proof}. 
We use a representation of $\FQ$ as the
subfield of a rational function field in one inderminate fixed under
an automorphism of order~$2$ and a degree map on a subring $\calQ$ of
the split algebra $Q_\FQ$ obtained from $Q$ by scalar extension from
$k$ to $\FQ$. A general method to attach an ordinary quadratic form
over a field to any generalized quadratic form on a split quaternion
algebra is given in~\S\,\ref{subsec:Morita}. It is used to reduce the
question of anisotropy of a generalized quadratic form on $Q_\FQ$
extended from $Q$ to the same question for an ordinary quadratic form
$q$ over $\FQ$. We use 
the degree map on $\calQ$ to define a degree for $q$-isotropic
vectors. A degree reduction argument then shows that the minimal
degree of an isotropic vector is~$1$, which means that isotropic
vectors of minimal degree come from isotropic vectors of the
generalized quadratic form over $Q$. This technique of proof is
inspired by Rost's  proof in~\cite{R} of the excellence property of
$\FQ$ for quadratic forms over $k$.

The first section is a survey of generalized quadratic forms focused
on the case of quaternion algebras. In its last subsection, it gives
a detailed description of the close
relationship between generalized quadratic forms and quadratic pairs,
to illuminate the equivalence between the viewpoint of Becher--Dolphin
in~\cite{BD} (or in~\cite{BoI}) and ours. As a result of this
equivalence, our main theorem can be rephrased as follows:

\begin{corol*} 
Let $A$ be a central simple algebra Brauer-equivalent to a quaternion algebra $Q$ over an arbitrary field $k$. Every anisotropic quadratic pair on $A$ remains anisotropic after scalar extension to the function field $F$ of the Severi-Brauer variety of $Q$, hence also after scalar extension to any odd degree field extension of $k$. 
\end{corol*}

Quadratic pairs are not used anywhere else in the paper.
The second section yields a description of
the function field $\FQ$ and of the degree map on~$\calQ$, and the
proof of the theorem is given in~\S\,\ref{sec:proof}.

\subsection*{Notation}
We refer to~\cite{EKM}, \cite{Knus} or \cite{BoI} for background on
central simple algebras with involution, quadratic forms,
Severi--Brauer varieties, and for any undefined terminology. 
We write
$\Trd_A$ for the reduced trace linear map from a central simple
algebra $A$ to its center. For any quaternion algebra $Q$, we write
$\invo$ for the canonical involution mapping $x\in Q$ to $\overline
x=\Trd_Q(x)-x$. 

\section{Generalized quadratic forms and quadratic pairs}
\label{sec:qf}

A general notion of quadratic form was introduced by Tits
in~\cite{Tits} in order to 
describe algebraic groups of type~$\mathsf{D}$ in arbitrary characteristic.
This notion was expanded in
several ways, notably by Wall~\cite{Wall} and Bak~\cite{Bak}. Thorough
expositions can be found in the monographs by Scharlau~\cite{Sch},
Knus~\cite{Knus} and Tits--Weiss~\cite{TW}. In the geometric tradition
of~\cite{TitsBN} and \cite{TW}, pseudo-quadratic spaces are defined
over division rings, 
which makes them 
ineffective for our purpose since the base ring may no
longer be a division ring after an extension of scalars. Therefore,
the generalized quadratic spaces studied in this work are modeled on
the ``unitary spaces'' of~\cite{Knus}. After reviewing their
definition in~\S\,\ref{subsec:def}, we describe in~\S\,\ref{subsec:Morita}
the Morita equivalence needed to relate generalized quadratic spaces
to ordinary quadratic forms over fields in the case where the base
quaternion algebra is split. All the material in~\S\,\ref{subsec:def}
and~\S\,\ref{subsec:Morita} can be found in a much more general form
in~\cite{Knus}, but the special case of quaternionic spaces affords
substantial simplifications because the base ring is a central simple
algebra equipped with a canonical involution. The
final~\S\,\ref{subsec:quadpair} outlines the correspondence between
generalized quadratic forms over quaternion algebras and quadratic
pairs on central simple algebras of index $2$. Quadratic pairs were
introduced in~\cite{BoI}; they 
provide alternative descriptions of linear algebraic groups of
type~$\mathsf{D}$ and are well-suited to scalar extension. The
description of the quadratic pair associated to a generalized
quadratic form uses some standard identification, recalled
in~\S\,\ref{subsec:standard}.  
\medbreak

Throughout this section, $Q$ is a quaternion algebra over an arbitrary
field $k$, with canonical involution $\invo$. For every right (resp.\
left) $Q$-module $V$, the left (resp.\ right) $Q$-module $\overline V$
is defined by $\overline V = \{\overline v\mid v\in V\}$ with the
$Q$-module structure given by
\[
  x\,\overline v = \overline{v\overline x} \qquad\text{(resp.\
    $\overline v\,x = 
    \overline{\overline x v}$)} \qquad\text{for $x\in Q$ and $v\in
    V$.} 
\]

\subsection{Definitions}
\label{subsec:def}

Let $V$ be a finitely generated (projective\footnote{Since $Q$ is a
  right artinian simple ring, every finitely generated right
  $Q$-module is projective: see~\cite[(2.8)]{Lam}.}) right $Q$-module
and let 
$V^*=\Hom_Q(V,Q)$, with its natural left $Q$-module structure.
A \emph{sesquilinear map} on $V$ is a $k$-bilinear map $s\colon
V\times V\to Q$ such that
\[
  s(v\,x, w\,y) = \overline x\, s(v,w)\,y
  \qquad\text{for $v$, $w\in V$ and $x$, $y\in Q$.}
\]
Sesquilinear maps on $V$ may also be viewed as $Q$-module
homomorphisms $V\to \overline{V^*}$ by attaching to $s$ the map
$\widehat s\colon V\to \overline{V^*}$ carrying $v\in V$ to
$\overline{s(v,\mbox{\textvisiblespace})}$. 
We write $\Sesq(V)$ for the $k$-vector space of sesquilinear maps on
$V$ and define an operator $\op$ on $\Sesq(V)$ by letting
\[
  s^\op(v,w) = \overline{s(w,v)}\qquad \text{for $s\in\Sesq(V)$ and
    $v$, $w\in V$}. 
\]
The $k$-vector spaces $\Herm^+(V)$ of \emph{even hermitian forms} and
$\Herm^-(V)$ of \emph{even skew-hermitian forms} on $V$ are defined by
\[
  \Herm^+(V)=\{s+s^\op\mid s\in\Sesq(V)\} \quad\text{and}\quad
  \Herm^-(V)=\{s-s^\op\mid s\in\Sesq(V)\}.
\]
Thus, $\Herm^+(V)=\Herm^-(V)$ if $\charac k=2$. \emph{Generalized
  quadratic forms} on $V$ are defined to be the elements of the
quotient space
\[
  \Quad(V)=\Sesq(V)/\Herm^+(V).
\]
For $s\in\Sesq(V)$, we write $[s]$ for the image of $s$ in $\Quad(V)$
and define $h_s\in\Herm^-(V)$ and $q_s\colon V\to Q/k$ by
\[
  h_s=s-s^\op \qquad\text{and}\qquad
  q_s(v) = s(v,v)+k \text{ for $v\in V$.}
\]
Note that $h_s$ and $q_s$ depend only on $[s]$, because $\op$ is an
involution on $\Sesq(V)$ and $(t+t^\op)(v,v) =
\Trd_Q\bigl(t(v,v)\bigr)\in 
k$ for $t\in\Sesq(V)$ and $v\in V$.

A generalized quadratic form $[s]$ on $V$ is said to be
\emph{nonsingular} if $\widehat h_s\colon V\to \overline{V^*}$ is an
isomorphism, 
and it is said to be \emph{isotropic} if there is a nonzero element
$v\in V$ such that $q_s(v)=0$. Every such element is termed
\emph{isotropic}.
\medbreak

\emph{Generalized quadratic spaces} over $Q$ are defined to be pairs
$(V,[s])$ 
consisting of a finitely generated right $Q$-module $V$ and a
nonsingular generalized quadratic form $[s]\in\Quad(V)$. The
generalized quadratic space $(V,[s])$ is said to be \emph{isotropic}
if $[s]$ is 
isotropic, and \emph{anisotropic} if $[s]$ is not isotropic.

If $\charac k\neq2$, then $[s]=[\frac12h_s]$ because
$s-\frac12h_s=\frac12(s+s^\op)$; moreover $\frac12h_s$ is the only form in
$\Herm^-(V)$ representing $[s]\in\Quad(V)$. Hence, over such fields, generalized
quadratic spaces 
$(V,[s])$ may be viewed as skew-hermitian spaces $(V,\frac12h_s)$.

In arbitrary characteristic, the maps $h_s$ and $q_s$ associated to a
generalized quadratic space $(V,[s])$ over $Q$ satisfy the following conditions:
\begin{equation}
  \label{eq:cond1}
  q_s(v+w)-q_s(v)-q_s(w) = h_s(v,w)+k\qquad\text{for $v$, $w\in V$}
\end{equation}
and
\begin{equation}
  \label{eq:cond2}
  q_s(v\,x) = \overline x\,q_s(v)\,x \qquad\text{for $v\in V$ and
    $x\in Q$.}
\end{equation}
When $Q$ is a division algebra and $V$ is a finite-dimensional right
$Q$-vector space, every map $q\colon V\to Q/k$
satisfying~\eqref{eq:cond1} and \eqref{eq:cond2} for some
$h\in\Herm^-(V)$ has the form $q=q_s$ for some uniquely 
determined generalized quadratic form $[s]\in\Quad(V)$, as is easily
seen by using a base of $V$ as in~\cite[\S\,8.2]{TitsBN}. Therefore,
generalized quadratic spaces over a quaternion division algebra can
be 
alternatively defined just in terms of maps $q$
satisfying~\eqref{eq:cond1} and \eqref{eq:cond2}, see
Seip-Hornix's paper~\cite{SH}, which predates~\cite{Tits}, or ~\cite[\S\,11]{TW}. 

\subsection{Morita equivalence}
\label{subsec:Morita}

In this subsection, we assume $Q$ is split and show how to associate
an ordinary quadratic form on a $k$-vector space to each generalized
quadratic form over $Q$, in such a way that isotropic quadratic forms
correspond to isotropic generalized quadratic forms. Our construction
is a special case of the Morita equivalence discussed
in~\cite[I(9.4)]{Knus}.

Let $Q=\End_k(I)$ for some $2$-dimensional $k$-vector space $I$, and
let $a\colon I\times I\to k$ be a nonzero alternating bilinear
form. The form $a$ is uniquely determined up to a factor in $k^\times$
since $\dim_kI=2$, and its adjoint involution on $Q$ is symplectic; it
is therefore the canonical involution $\invo$. Viewing $I$ as a left
$Q$-module we thus have
\[
  a(\xi,x\,\eta)=a(\overline x\,\xi,\eta) \qquad\text{for $\xi$, $\eta\in I$ and
    $x\in Q$.}
\]
For any finitely generated right $Q$-module $V$ and any
$s\in\Sesq(V)$, we define a $k$-bilinear map $a*s$ on $V\otimes_QI$
by 
\[
  (a*s)(v\otimes\xi, w\otimes\eta) = a(\xi, s(v,w)\eta)
  \qquad\text{for $v$, $w\in V$ and $\xi$, $\eta\in I$,}
\]
see~\cite[I.8]{Knus}. Using this, we get the following: 
\begin{prop}
\label{prop:nonsing}
(a) The map $q_{a*s}\colon V\otimes_QI\to k$, defined by 
\[
  q_{a*s}(v\otimes\xi) = a(\xi, s(v,v)\,\xi)
  \qquad\text{for $v\in V$ and $\xi\in I$}
\]
is a quadratic form on $V\otimes_QI$. 
It depends only on the choice of $a$ and the class $[s]$ of the sesquilinear form $s$. \\
(b) The quadratic form $q_{a*s}$ is nonsingular (resp.\ isotropic) if
  and only if the generalized quadratic form $[s]$ is nonsingular
  (resp.\ isotropic).
\end{prop}

\begin{proof}
(a) The map $q_{a*s}$ is obtained by restricting the $k$-bilinear map $a*s$ to diagonal elements. Therefore, it is a quadratic form, with associated symmetric bilinear form \[b_{a*s}\colon (V\otimes_QI) \times (V\otimes_QI)\to k\] defined by 
\begin{align*}
  b_{a*s}(v\otimes \xi, w\otimes\eta) &
        = (a*s)(v\otimes\xi, w\otimes\eta) + (a*s)(w\otimes\eta,
                                        v\otimes\xi)\\
                                      & = (a*h_s)(v\otimes\xi,w\otimes\eta)
                                        \quad\text{ for $v$, $w\in V$
                                        and $\xi$, $\eta\in I$.}
\end{align*}
Since $a$ is alternating, $q_{a*s}=0$ for all
$s\in\Herm^+(V)$, hence $q_{a*s}$ depends only on $a$ and
$[s]\in\Quad(V)$.

(b) The left $Q$-module structure on $I$ induces a right $Q$-module
  structure on its $k$-dual $I^\sharp = \Hom_k(I,k)$: for $\psi\in
  I^\sharp$, $x\in Q$ and $\xi\in I$,
  \[
    (\psi x)(\xi) = \psi(x\xi).
  \]
  Let also $(V\otimes_QI)^\sharp = \Hom_k(V\otimes_QI, k)$. As explained in~\cite[I(8.2.1)]{Knus}, 
a canonical isomorphism of $k$-vector spaces
  \begin{equation}
    \label{eq:isodual}
    I^\sharp\otimes_QV^* \xrightarrow{\sim} (V\otimes_QI)^\sharp
  \end{equation}
  identifies $\psi\otimes\varphi\in I^\sharp\otimes_QV^*$ with the
  linear functional
  \[
    \psi\otimes\varphi\colon v\otimes\xi \mapsto \psi(\varphi(v)\xi)
    \qquad\text{for $v\in V$ and $\xi\in I$.}
  \]
  Since $b_{a*s} = a*h_s$, we have, under this identification,
  \[
    b_{a*s}(v\otimes\xi,\mbox{\textvisiblespace}) = a(\xi,
    \mbox{\textvisiblespace}) \otimes h_s(v, \mbox{\textvisiblespace})
    \qquad\text{for $v\in V$ and $\xi\in I$.}
  \]
  The map $\xi\mapsto a(\xi,\mbox{\textvisiblespace})$ is an
  isomorphism of $k$-vector spaces $I\to I^\sharp$ because $a$ is
  nonzero, hence the map $v\otimes\xi\mapsto b_{a*s}(v\otimes\xi,
  \mbox{\textvisiblespace})$ is an isomorphism $V\otimes_QI \to
  (V\otimes_QI)^\sharp$ if and only if $\widehat h_s$ is an
  isomorphism. This shows that $q_{a*s}$ is nonsingular if and
  only if $[s]$ is nonsingular, and it only remains to prove the isotropy statement. 

  If $v\in V$ is such that $s(v,v)\in k$, then
  $q_{a*s}(v\otimes\xi)=0$ for all $\xi\in I$. Hence $q_{a*s}$ is
  isotropic if $[s]$ is isotropic. To prove the converse, first observe that if
  $\xi\in I$ is nonzero then $I=Q\xi$, hence every element in
  $V\otimes_QI$ has the form $v\otimes\xi$ for some $v\in V$. Suppose
  $v\otimes\xi$ is isotropic for $q_{a*s}$, and pick $\xi'\in I$ such
  that $\xi$, $\xi'$ is a $k$-base of $I$. Since $Q=\End_F(I)$, there
  exists $x\in Q$ such that $x\xi=x\xi'=\xi$. Then
  $v\otimes\xi=vx\otimes\xi$, hence $vx\neq0$. Since
  $a(\xi,s(v,v)\xi)=q_{a*s}(v\otimes\xi)=0$ it follows that for all
  $\lambda$, $\lambda'\in k$
  \[
    a\bigl(\xi\lambda+\xi'\lambda',
    s(vx,vx)(\xi\lambda+\xi'\lambda')\bigr) =
    a\bigl(x(\xi\lambda+\xi'\lambda'),\,
    s(v,v)x(\xi\lambda+\xi'\lambda')\bigr)  =
    a(\xi,s(v,v)\xi)(\lambda+\lambda')^2=0.
  \]
  Therefore, $s(vx,vx)\eta\in\eta k$ for all $\eta\in I$; this implies
  $s(vx,vx)\in k$, hence $vx$ is isotropic for $[s]$.
\end{proof}

\begin{remk}
Proposition~\ref{prop:nonsing}(b) also readily follows from the fact that
the choice of $a$ induces a Morita equivalence between the categories of generalized quadratic
spaces over $Q$ and of (ordinary) quadratic spaces over $k$ which carries
any generalized quadratic space $(V,[s])$ over $Q$ to the quadratic
space $(V\otimes_QI, q_{a*s})$ over $k$, see~\cite[I.9]{Knus}. We will
not need this more general fact.
\end{remk}

\subsection{Standard identifications}
\label{subsec:standard}
The description of quadratic pairs associated to generalized quadratic
forms uses some standard identifications, see~\cite[\S\,5.A]{BoI}. As
a preparation for the next section, we compare these identifications
for a generalized quadratic form and for the Morita equivalent
quadratic form defined in \S\,\ref{subsec:Morita}.  

Given a nonsingular hermitian form $h$ on the finitely generated right
$Q$-module $V$, the map 
$\overline v\to 
h(v,\mbox{\textvisiblespace})$ defines an isomorphism of left
$Q$-modules $\overline V\to V^*$, which is a twisted version of
$\widehat h$. It induces a $k$-vector space isomorphism
\begin{equation}
  \label{eq:Phi}
  \Phi_{h}\colon V\otimes_Q\overline V \xrightarrow{\sim}
  V\otimes_QV^* = \End_Q(V), \qquad
  v\otimes\overline w\mapsto v\,h(w,\mbox{\textvisiblespace})
  \quad\text{for $v$, $w\in V$.}
  \end{equation}
Similarly, given a nonsingular quadratic form $q\colon W\to k$ on a
finite dimensional $k$-vector space $W$ 
with polar bilinear form $b$, the map $w\mapsto
b(w,\mbox{\textvisiblespace})$ defines an isomorphism from $W$ to its
dual $W^\sharp=\Hom_k(W,k)$, hence also a $k$-vector space
isomorphism
\begin{equation}
\label{eq:Phisplit}
  \Phi_b\colon W\otimes_kW \xrightarrow{\sim} W\otimes_kW^\sharp =
  \End_k(W), \qquad
  v\otimes w \mapsto v\,b(w,\mbox{\textvisiblespace})
  \quad\text{for $v$, $w\in W$.}
\end{equation}

We assume $Q$ is split and use the same notation as
in~\S\,\ref{subsec:Morita}. In particular,  
$(V,[s])$ is a generalized quadratic space over $Q=\End_k(I)$, $a:\,I\times I\rightarrow k$ is a fixed nonzero alternating bilinear form, and we use it to associate to $[s]$ a quadratic form $q_{a* s}$ on $V\otimes_k I$. Applying~\eqref{eq:Phi} and~\eqref{eq:Phisplit} to $h_s$ and $b_{a*s}$, we get isomorphisms 
\[\Phi_{h_s}\colon V\otimes_Q\overline V\rightarrow \End_Q(V),\] 
and 
\[
  \Phi_{b_{a*s}}\colon (V\otimes _Q I)\otimes _k(V\otimes_Q
  I)\rightarrow \End_k(V\otimes_Q I).
\]
Moreover, since $Q=\End_k(I)$, there is a canonical isomorphism 
\begin{equation}
  \label{eq:Thetadef}
  \Theta\colon \End_Q(V)\xrightarrow{\sim} \End_k(V\otimes_QI)
\end{equation}
such that $
  \Theta(\rho)(v\otimes\xi) = \rho(v)\otimes\xi$ for
    $\rho\in\End_Q(V)$, $v\in V$ and $\xi\in I$.

Define a $k$-vector space isomorphism $\Theta'\colon
  (V\otimes_QI)\otimes_k(V\otimes_QI) \to V\otimes_Q\overline V$ by
  using the isomorphism $V\otimes_QI\to (V\otimes_QI)^\sharp$ mapping
  $v\otimes\xi$ to $b_{a*s}(v\otimes\xi,\mbox{\textvisiblespace})$
  together
  with the isomorphism $I^\sharp\otimes_QV^*\to(V\otimes_QI)^\sharp$
  of~\eqref{eq:isodual}, the identification $I\otimes_kI^\sharp =
  \End_k(I)=Q$ and the isomorphism $\overline V\to V^*$ mapping
  $\overline v$ to $h_s(v,\mbox{\textvisiblespace})$. More precisely, 
  $\Theta'$ is the following composition
  \[
    (V\otimes_QI)\otimes_k(V\otimes_QI) \to
    (V\otimes_QI)\otimes_k(V\otimes_QI)^\sharp \to
    V\otimes_QI\otimes_kI^\sharp\otimes_QV^* \to V\otimes_QV^* \to
    V\otimes_Q\overline V.
  \]
  Thus, for $v_1$, $v_2\in V$ and $\xi_1$, $\xi_2\in I$, we have 
  \[
    \Theta'\bigl((v_1\otimes\xi_1)\otimes(v_2\otimes\xi_2)\bigr) =
    \bigl(v_1\cdot(\xi_1\otimes
    a(\xi_2,\mbox{\textvisiblespace}))\bigr) 
    \otimes \overline v_2,
  \]
 where $\xi_1\otimes a(\xi_2,\mbox{\textvisiblespace})\in
  I\otimes_kI^\sharp = Q$. 
  
  \begin{lemma}
  \label{lem:diag}
  The following square of isomorphisms is commutative: 
   \[
      \xymatrix{
        (V\otimes_QI)\otimes_k(V\otimes_QI) \ar[rr]^{\Phi_{b_{a*s}}}
        \ar[d]_{\Theta'} && \End_k(V\otimes_QI)\\
        V\otimes_Q\overline V \ar[rr]^{\Phi_{h_s}} && \End_Q(V)
        \ar[u]_{\Theta}
      }\]
      Moreover, for all $v$, $w\in V$, we have 
\[  \Trd_{\End_Q(V)}\bigl(\Phi_{h_s}(v\otimes \overline w)\bigr) =
  \Trd_Q\bigl(h_s(w,v)\bigr). \]
  \end{lemma} 
  \begin{remk}
  The second assertion also holds when $Q$ is a division algebra
  by~\cite[(5.1)]{BoI}.  
  \end{remk} 
  
  \begin{proof}
  For $v_1$, $v_2\in V$ and $\xi_1$, $\xi_2\in I$, the endomorphism 
  \[(\Theta\circ\Phi_{h_s}\circ
  \Theta')\bigl((v_1\otimes\xi_1)\otimes(v_2\otimes\xi_2)\bigr)\] maps
  $w\otimes\eta\in V\otimes_QI$ to
  \[
    \bigl(v_1\cdot(\xi_1\otimes
    a(\xi_2,\mbox{\textvisiblespace}))h_s(v_2,w)\bigr) \otimes\eta =
    v_1\otimes\bigl((\xi_1\otimes a(\xi_2,\mbox{\textvisiblespace}))
    h_s(v_2,w)\cdot\eta\bigr).
  \]
  Now, $(\xi_1\otimes a(\xi_2,\mbox{\textvisiblespace}))
    h_s(v_2,w)\cdot\eta=\xi_1\otimes a(\xi_2,h_s(v_2,w)\eta) =
    \xi_1\otimes 
    (a*h_s)(v_2\otimes\xi_2,w\otimes\eta)$. Since $b_{a*s} = a*h_s$,
    it follows that
    \[
      \bigl(v_1\cdot(\xi_1\otimes
      a(\xi_2,\mbox{\textvisiblespace}))h_s(v_2,w)\bigr) \otimes\eta =
      (v_1\otimes\xi_1)\otimes b_{a*s}(v_2\otimes\xi_2,w\otimes\eta).\]
      This proves the first assertion of the lemma. 
     
         Since the reduced trace is linear and $\Theta'$ is an isomorphism, it is enough to prove the second assertion for the image of a symbol $(v_1\otimes \xi_1)\otimes (v_2\otimes \xi_2)$. So, 
        we let 
  $v=v_1\cdot\bigl(\xi_1\otimes
    a(\xi_2,\mbox{\textvisiblespace})\bigr)$ and $w=v_2,$ so that $v\otimes\overline w=\Theta'\bigl((v_1\otimes \xi_1)\otimes(v_2\otimes \xi_2)\bigr).$ By uniqueness of the reduced trace, we have $\Trd_{\End_Q(V)}=\Trd_{\End_k(V\otimes_QI)}\circ \Theta$. Therefore, the commutative diagram above shows that 
    \[\Trd_{\End_Q(V)}(\Phi_{h_s}(v\otimes \overline w))=\Trd_{\End_k(V\otimes_Q I)}\bigl(\Phi_{b_{a*s}}((v_1\otimes\xi_1)\otimes (v_2\otimes \xi_2))\bigr).\] 
By~\cite[(5.1)]{BoI}, this reduced trace is equal to $b_{a*s}(v_2\otimes \xi_2,v_1\otimes \xi_1)=a(\xi_2,h_s(v_2,v_1)\xi_1)$. 
On the other hand, we have 
\[h_s(w,v)=h_s(v_2,v_1\cdot \bigl(\xi_1\otimes
    a(\xi_2,\mbox{\textvisiblespace})\bigr)=h_s(v_2,v_1)\bigl(\xi_1\otimes
    a(\xi_2,\mbox{\textvisiblespace})\bigr)=\bigl(h_s(v_2,v_1)\cdot\xi_1\bigl)\otimes
    a(\xi_2,\mbox{\textvisiblespace}).
\]
The lemma follows, since under the identifications $I\otimes_kI^\sharp
=\End_k(I)=Q$,
\begin{equation}
  \label{eq:Trdsplit}
  \Trd_Q(\xi\otimes\psi) = \psi(\xi) \qquad\text{for $\xi\in I$ and
    $\psi\in I^\sharp$.}
\end{equation}
\qedhere
\end{proof}

\subsection{Quadratic pairs}
\label{subsec:quadpair}

Recall from~\cite[\S\,5.B]{BoI} that a quadratic pair on a central
simple algebra $A$ of even degree~$n$ over an arbitrary field $k$ is a
pair $(\sigma,f)$ consisting of a $k$-linear involution $\sigma$ and a
linear map $f\colon\Sym(\sigma)\to k$, where $\Sym(\sigma) = \{x\in
A\mid \sigma(x)=x\}$, subject to the following conditions:
\begin{enumerate}
\item[(i)]
  $\dim_k\Sym(\sigma)=\frac12n(n+1)$ and $\Trd_A(x)=0$ for all $x\in
  A$ such that $\sigma(x)=-x$;
\item[(ii)]
  $f\bigl(x+\sigma(x)\bigr) = \Trd_A(x)$ for all $x\in A$.
\end{enumerate}
The map $f$ is called the \emph{semitrace} of the quadratic pair
$(\sigma,f)$, because~(ii) implies $f(x)=\frac12\Trd_A(x)$ for all
$x\in\Sym(\sigma)$ if $\charac k\neq2$.

When $A=\End_k(W)$ for some $k$-vector space of even dimension, it is
shown in~\cite[(5.11)]{BoI} how to attach a quadratic pair
$(\sigma_b,f_q)$ to each nonsingular quadratic form $q\colon W\to k$
with polar bilinear form $b$ using the isomorphism~\eqref{eq:Phisplit}. 
The involution $\sigma_b$ is adjoint to $b$, and is given by the
switch map on $W\otimes_kW$:
\[
  \sigma_b\circ\Phi_b(v\otimes w) = \Phi_b(w\otimes v)
  \qquad\text{for $v$, $w\in W$.}
\]
Therefore, $\Sym(\sigma)$ is spanned by elements of the form
$\Phi_b(w\otimes w)$ with $w\in W$.
The semitrace $f_q$ is the unique linear map such that
$f_q\bigl(\Phi_b(w\otimes w)\bigr) = q(w)$ for all $w\in W$. Moreover,
every quadratic pair on $A$ is 
attached to a nonsingular quadratic form on $W$, uniquely determined
up to a factor in $k^\times$.
\medbreak

Our goal in this subsection is to show how to associate a quadratic
pair on $\End_Q(V)$ to every generalized quadratic space $(V,[s])$ on
$Q$. We also show that when $Q=\End_k(I)$ the quadratic
pair associated to $[s]$ coincides with the quadratic pair attached to
$(V\otimes_QI,q_{a*s})$ under the canonical isomorphism
$\Theta\colon\End_Q(V)\xrightarrow{\sim}\End_k(V\otimes_QI)$
of~\eqref{eq:Thetadef}. Thus, the Morita equivalence discussed 
in~\S\,\ref{subsec:Morita} has no effect on quadratic pairs.
\medbreak

Let $(V,[s])$ be a generalized quadratic space over $Q$ and let $\Phi_{h_s}$ be the isomorphism defined in~\eqref{eq:Phi}. 
 Write $\sw$ for the switch map on $V\otimes_Q\overline V$, which
carries $v\otimes\overline w$ to $w\otimes\overline{v}$ for $v$, $w\in
V$. The involution $\sigma_{h_s}$ adjoint to $h_s$ satisfies
$\sigma_{h_s}\circ\Phi_{h_s} = -\Phi_{h_s}\circ\sw$ because $h_s$ is
skew-hermitian, hence the space $\Skew(\sw)=\{u\in V\otimes_Q\overline
V\mid \sw(u)=-u\}$ is mapped bijectively to $\Sym(\sigma_{h_s})$ by
$\Phi_{h_s}$:
\[
  \Sym(\sigma_{h_s}) = \Phi_{h_s}\bigl(\Skew(\sw)\bigr).
\]
For every $t\in\Sesq(V)$, define a $k$-linear functional
\[
  T_t\colon V\otimes_Q\overline V \to k
  \qquad\text{by}\quad
  T_t(v\otimes\overline w) = \Trd_Q\bigl(t(w,v)\bigr)
  \quad\text{for $v$, $w\in V$.}
\]
Then $T_{t^\op} = T_t\circ\sw$ for all $t\in\Sesq(V)$, hence
$T_{t+t^\op}$ 
vanishes on $\Skew(\sw)$ and the restriction of $T_s$ to $\Skew(\sw)$
depends only on $[s]\in\Quad(V)$. Let $f_{[s]}\colon
\Sym(\sigma_{h_s})\to k$ be defined by
\[
  f_{[s]}\bigl(\Phi_{h_s}(u)\bigr) = T_s(u) \qquad\text{for $u\in
    \Skew(\sw)$.}
\]

\begin{prop}
  \label{prop:quadpair}
  The pair $(\sigma_{h_s},f_{[s]})$ is a quadratic pair on $\End_Q(V)$
  attached to $[s]\in\Quad(V)$. Every quadratic pair on $\End_Q(V)$ is
  attached to a nonsingular generalized quadratic form on $V$, which
  is uniquely determined up to a factor in $k^\times$.
\end{prop}

\begin{proof}
  Since $h_s$ is an even skew-hermitian form on $V$, the adjoint
  involution $\sigma_{h_s}$ satisfies condition~(i) in the definition
  of a quadratic pair by~\cite[(2.6)\,\&\,(4.2)]{BoI}. To check condition~(ii),
  observe that for $v$, $w\in V$ we have
\begin{multline*}
  f_{[s]}\bigl(\Phi_{h_s}(v\otimes\overline w) +
  \sigma_{h_s}(\Phi_{h_s}(v\otimes\overline w))\bigr) =
  T_s(v\otimes\overline w) - T_s(w\otimes\overline v)
  \\
  =\Trd_Q\bigl(s(w,v)\bigr) - \Trd_Q\bigl(s(v,w)\bigr) =
  \Trd_Q\bigl(h_s(w,v)\bigr). 
\end{multline*}
Using \cite[(5.1)]{BoI} and Lemma~\ref{lem:diag}, we see that for $v$, $w\in V$
\[
  f_{[s]}\bigl(\Phi_{h_s}(v\otimes\overline w) +
  \sigma_{h_s}(\Phi_{h_s}(v\otimes\overline w))\bigr) =
  \Trd_{\End_Q(V)}\bigl(\Phi_{h_s}(v\otimes\overline w)\bigr),
\]
hence $(\sigma_{h_s},f_{[s]})$ is a quadratic pair on $\End_Q(V)$.
\medbreak

Now, let $(\sigma,f)$ be a quadratic pair on
$\End_Q(V)$. By~\cite[(4.2)]{BoI}, the involution $\sigma$ is adjoint
to some nonsingular skew-hermitian form $h$ on $V$. 
To prove the result, we need to construct a sesquilinear form $s$ on $V$ such that $f=f_{[s]}$ and $h=h_s$.
Consider the
$k$-vector space isomorphism $\Phi_h\colon V\otimes_Q\overline V \to
\End_Q(V)$ defined as in~\eqref{eq:Phi}, and recall
from~\cite[(5.7)]{BoI} that there exists $\ell\in\End_Q(V)$ such that
\begin{equation}
  \label{eq:fu}
  f\bigl(\Phi_h(u)\bigr) = \Trd_{\End_Q(V)}\bigl(\ell\,\Phi_h(u)\bigr)
  \qquad\text{for all $u\in\Skew(\sw)$.}
\end{equation}
To every $v$, $w\in V$ we attach the $k$-linear functional
$\tau_{v,w}\colon Q\to k$ defined by
\[
  \tau_{v,w}(x) = \Trd_{\End_Q(V)}\bigl(\ell\,\Phi_{h}(v\,x\otimes
  \overline w)\bigr).
\]
Since the bilinear form $(X,Y)\mapsto\Trd_Q(XY)$ on $Q$ is
nonsingular, there exists a unique element $s(w,v)\in Q$ such that
\[
  \tau_{v,w}(x)=\Trd_Q(s(w,v)x)\qquad\text{for all $x\in Q$.}
\]
Consider the map $s\colon V\times V\to Q$ that carries $(w,v)\in
V\times V$ to $s(w,v)$.
It is readily verified that $s\in\Sesq(V)$, and by definition
\[
  \Trd_{\End_Q(V)}\bigl(\ell\,\Phi_h(v\otimes\overline w)\bigr) =
  \Trd_Q\bigl(s(w,v)\bigr)=T_s(v\otimes\overline w) \qquad\text{for
    $v$, $w\in V$,}
\]
hence $\Trd_{\End_Q(V)}\bigl(\ell\,\Phi_h(u)\bigr) = T_s(u)$ for all
$u\in V\otimes_Q\overline V$. 
Comparing with~\eqref{eq:fu}, we see
that $f\circ\Phi_h$ coincides with $T_s$ on $\Skew(\sw)$, and it only remains to prove that $h=h_s$.
Using condition~(ii) in the definition of a quadratic pair, together with ~\cite[(5.1)]{BoI} and Lemma~\ref{lem:diag}, we obtain for $v$, $w\in V$
\[
  f\bigl(\Phi_h(v\otimes\overline w) + \sigma(\Phi_h(v\otimes
  \overline w))\bigr) = \Trd_{\End_Q(V)}\bigl(\Phi_h(v\otimes
  \overline w)\bigr) = \Trd_Q\bigl(h(w,v)\bigr).
\]
But $\sigma\bigl(\Phi_h(v\otimes\overline w)\bigr) =
-\Phi_h(w\otimes\overline v)$ and $f\circ\Phi_h = T_s$ on
$\Skew(\sw)$, hence we also have for $v$, $w\in V$
\[
  f\bigl(\Phi_h(v\otimes\overline w) + \sigma(\Phi_h(v\otimes
  \overline w))\bigr) = T_s(v\otimes\overline w) -
  T_s(w\otimes\overline v) = \Trd_Q\bigl(h_s(w,v)\bigr).
\]
Therefore, $\Trd_Q\bigl(h(w,v)\bigr) = \Trd_Q\bigl(h_s(w,v)\bigr)$ for
all $v$, $w\in V$ and also, since $h$ and $h_s$ are sesquilinear,
$\Trd_Q\bigl(h(w,v)x\bigr) = \Trd_Q\bigl(h_s(w,v)x\bigr)$ for all
$x\in Q$. It follows that $h(w,v)=h_s(w,v)$ for $v$, $w\in V$ because
the bilinear form $(X,Y)\mapsto \Trd_Q(XY)$ on $Q$ is
nonsingular. Therefore, $(\sigma,f)=(\sigma_{h_s},f_{[s]})$.

It remains to prove $[s]$ is unique up to a scalar factor. Consider
another sesquilinear form $s'\in \Sesq(V)$ and assume it gives rise to
the same quadratic pair. In particular, the skew-hermitian forms
$h_s=s-s^\op$ and $h_{s'}=s'-s'^\op$ induce the same adjoint
involutions. Hence they are equal up to a scalar factor. Scaling $s'$
if necessary, we may thus assume $h_s=h_{s'}$. It follows that $s$ and
$s'$ differ by a hermitian form $g$, and it only remains to prove that
$g$ is actually even-hermitian. By definition of the corresponding
semitraces, $T_s$ and $T_{s'}$ coincide on $\Skew(\sw)$, therefore
$T_g$ vanishes on this subset. In particular, for all $v\in V$ and
$x\in \Skew(Q,\invo)$, we have \[T_g(vx\otimes\overline
  v)=\Trd_Q(g(v,v)x)=0.\] By~\cite[(2.3)]{BoI}, we get that $g(v,v)$
is a symmetrized element of $(Q,\invo)$ for all $v\in V$ and the
result follows by~\cite[I(3.1.1)]{Knus}.  
\end{proof}

Recall from~\cite[(6.5)]{BoI} that a quadratic pair $(\sigma,f)$ on a
central simple $k$-algebra $A$ is said to be \emph{isotropic} if there
exists a nonzero right ideal $I\subset A$ subject to the following
conditions:
\begin{enumerate}
\item[(i)]
  $\sigma(x)y=0$ for all $x$, $y\in I$;
\item[(ii)]
  $f(x)=0$ for all $x\in I\cap\Sym(\sigma)$.
\end{enumerate}
It is shown in~\cite[(6.6)]{BoI} that the quadratic pair
$(\sigma_b,f_q)$ on $\End_k(W)$ attached to a nonsingular quadratic
form $q$ on a $k$-vector space $W$ is isotropic if and only if $q$ is
isotropic. We prove the same property for generalized quadratic
forms.

\begin{prop}
  \label{prop:isot}
  Let $(V,[s])$ be a generalized quadratic space over $Q$ and
  $(\sigma_{h_s},f_{[s]})$ the attached quadratic pair on
  $\End_Q(V)$. The form $[s]$ is isotropic if and only if the pair
  $(\sigma_{h_s}, f_{[s]})$ is isotropic.
\end{prop}

\begin{proof}
  Every right ideal in $\End_Q(V)$ has the form $\Hom_Q(V,U)$ for some
  $Q$-submodule $U\subset V$. Using the isomorphism $\Phi_{h_s}$
  of~\eqref{eq:Phi}, we may write $\Hom_Q(V,U) = \Phi_{h_s}(U\otimes_Q
  \overline V)$. Suppose $\Hom_Q(V,U)$ is a nonzero ideal
  satisfying~(i) and (ii), and let $u\in U$ be nonzero. For all $x\in
  Q$ such that $\overline x=-x$, we have $\Phi_{h_s}(u\,x\otimes
  \overline u) \in \Hom_Q(V,U)\cap \Sym(\sigma_{h_s})$ and
  \begin{equation}
    \label{eq:Trsu}
    \Trd_Q(s(u,u)x) = f_{[s]}\bigl(\Phi_{h_s}(u\,x\otimes \overline
    u)\bigr) = 0.
  \end{equation}
  Therefore, $s(u,u)\in k$ and $u$ is isotropic in $(V,[s])$.

  Conversely, let $u\in V$ be isotropic, and consider the nonzero
  right ideal \[
    \Hom_Q(V,uQ)=\Phi_{h_s}(uQ\otimes_Q\overline V).
  \] 
  It satisfies~(i), because
  for all $v$, $w\in V$
  \[
    \sigma_{h_s}\bigl(\Phi_{h_s}(u\otimes \overline v)\bigr)\circ
    \Phi_{h_s}(u\otimes \overline w) = -\Phi_{h_s}(v\otimes \overline
    u) \circ 
    \Phi_{h_s}(u\otimes\overline w) = - \Phi_{h_s}(v\,h_s(u,u) \otimes
    \overline w)
  \]
  and $h_s(u,u) = s(u,u)-\overline{s(u,u)} = 0$.

  To prove (ii), we need to show that $T_s$ vanishes on
  $(uQ\otimes_Q\overline V)\cap \Skew(\sw)$ since $f_{[s]}\circ
  \Phi_{h_s} = T_s$ by definition of $f_{[s]}$.
  We first observe that  $(uQ\otimes_Q\overline V)\cap\Skew(\sw)$ has
  a nice description.   
  Since $\widehat h_s$ is an isomorphism of right $Q$-modules, it
follows that for all $x\in Q$ the equations $ux=0$ and
$\widehat h_s(u)x=0$ are equivalent. Therefore, the right ideal
$A(u)=\{x\in Q\mid ux=0\}$ can be alternatively described as
\[
  A(u)=\{x\in Q\mid \overline x h_s(u,w)=0 \text{ for all $w\in V$}\},
\]
which means that the left ideal $\overline{A(u)}$ image of $A(u)$
under the canonical map $\invo$ is the annihilator of the right ideal
$h_s(u,V)\subset Q$:
\[
  \overline{A(u)} = \{y\in Q\mid yh_s(u,V)=0\}.
\]
By~\cite[(1.14)]{BoI}, it follows that $h_s(u,V)$ is the annihilator
of the left ideal $\overline{A(u)}$:
\[
  h_s(u,V) = \{x\in Q\mid yx=0 \text{ for all $y\in
    \overline{A(u)}$}\}.
\]
If $e\in Q$ is an idempotent such that $A(u)=eQ$, then $y(1-\overline
e)=0$ for all $y\in\overline{A(u)}$, hence there exists $w\in V$ such
that $h_s(u,w)=1-\overline e$.

Now, let $v\in V$ be such that $u\otimes\overline v=-v\otimes
\overline u$, so that $u\otimes \overline v\in (uQ\otimes_Q\overline V)\cap\Skew(\sw)$. 
Using $\Phi_{h_s}$, we see that this equation implies
that $uh_s(v,w) = -vh_s(u,w)$, hence $v(1-\overline e) =
-uh_s(v,w)$. But since $ue=0$,
\[
  v(1-\overline e)\otimes\overline u = v\otimes \overline u- v\otimes
  \overline{ue} = v\otimes\overline u.
\]
Therefore, $v\otimes \overline u=-u h_s(v,w)\otimes\overline u$. We
have thus proved
\[
  (uQ\otimes_Q \overline V)\cap \Skew(\sw) \subset \{ux\otimes
  \overline u\mid x\in Q\}.
\]
Note that $u=u(1-e)$ since $ue=0$, hence
\[
  \{ux\otimes \overline u\mid x\in
  Q\} = \{u(1-e)x(1-\overline e)\otimes \overline u \mid x\in Q\}
  = \{ux\otimes \overline u\mid x\in (1-e)Q(1-\overline e)\}.
\]
Moreover, $ux\otimes \overline u = 0$ implies $uxh_s(u,w)=0$, i.e.,
$ux(1-\overline e)=0$. It follows that $ux=0$ if $x\in
(1-e)Q(1-\overline e)$. But then $x\in A(u)=eQ$, hence $x=0$ if $x\in
(1-e)Q(1-\overline e)$. Thus, every element in $\{ux\otimes\overline
e\mid x\in Q\}$ has a unique expression of the form $ux\otimes \overline
u$ with $x\in (1-e)Q(1-\overline e)$. Therefore,
\[
  (uQ\otimes_Q\overline V)\cap \Skew(\sw) = \{ux\otimes \overline
  u\mid x\in 
  (1-e)Q(1-\overline e) \text{ and }\overline x=-x\}.
\]
Since $u$ is isotropic, we have $s(u,u)\in k$, hence for all $x\in Q$
such that $\overline x = -x$,
\[
  T_s(ux\otimes \overline u) = \Trd_Q\bigl(s(u,ux)) = s(u,u)
  \Trd_Q(x) =0.
\]
Therefore, $T_s$ vanishes on $(uQ\otimes_Q\overline V)\cap\Skew(\sw)$,
which means that $\Hom_Q(V,uQ)$ satisfies condition~(ii) and completes
the proof that it is an isotropic ideal, as required.
\end{proof}

For the rest of this section, we assume that $Q$ is split, and use the
same notation as in \S\,\ref{subsec:Morita} and
\S\,\ref{subsec:standard}.  
In particular, $\Theta$ is the isomorphism
$\End_Q(V)\xrightarrow{\sim} \End_k(V\otimes_QI)$
of~\eqref{eq:Thetadef} and $q_{a*s}$  is the ordinary quadratic form
on $V\otimes_QI$ associated to $[s]$.

\begin{prop}
  \label{prop:quadpairs}
  The canonical isomorphism $\Theta$ carries the quadratic pair
  $(\sigma_{h_s}, f_{[s]})$  attached to $[s]$ on $\End_Q(V)$ to the
  quadratic pair $(\sigma_{b_{a*s}},f_{q_{a*s}})$ attached to
  $q_{a*s}$ on $\End_k(V\otimes_QI)$.
\end{prop}

\begin{proof}
  The adjoint involution $\sigma_{h_s}$ is defined by the equation
  \[
    h_s(\rho(v),w) = h_s\bigl(v,\sigma_{h_s}(\rho)(w)\bigr)
    \qquad\text{for $\rho\in\End_Q(V)$ and $v$, $w\in V$.}
  \]
  Since $b_{a*s}=a*h_s$, this equation implies that for
  $\rho\in\End_Q(V)$, $v$, $w\in V$ and $\xi$, $\eta\in I$,
  \[
    b_{a*s}(\rho(v)\otimes\xi, w\otimes\eta) =
    a(\xi,\,h_s(\rho(v),w)\eta) = a(\xi,\,
    h_s(v,\sigma_{h_s}(\rho)(w))\eta) = b_{a*s}(v\otimes\xi,
    \sigma_{h_s}(\rho)(w)\otimes\eta).
  \]
  Therefore, $\sigma_{b_{a*s}}\circ\Theta = \Theta\circ\sigma_{h_s}$,
  hence $\Theta$ maps $\Sym(\sigma_{h_s})$ to
  $\Sym(\sigma_{b_{a*s}})$. It remains to show that
  $f_{q_{a*s}}\bigl(\Theta(\rho)\bigr) = f_{[s]}(\rho)$ for $\rho\in
  \Sym(\sigma_{h_s})$.
  
  Recall that $\Sym(\sigma_{b_{a*s}})$ is spanned by the elements
  $\Phi_{b_{a*s}}\bigl((v\otimes \xi)\otimes (v\otimes \xi)\bigr)$,
  with $v\in V$ and $\xi\in I$. Moreover,  
  $f_{q_{a*s}}$ maps such an element to $q_{a*s}(v\otimes
  \xi)=a(\xi,s(v,v)\xi)$.  
Therefore, by the commutative diagram given in Lemma~\ref{lem:diag},
we need to prove that for $v\in V$ and $\xi\in I$
    \[
       f_{[s]}\bigl((\Phi_{h_s}\circ\Theta')(v\otimes\xi,
      v\otimes\xi)\bigr)=a(\xi,s(v,v)\xi). 
    \]
  Since $\Theta'(v\otimes\xi, v\otimes\xi) =
     v\cdot\bigl(\xi\otimes a(\xi,
     \mbox{\textvisiblespace})\bigr)\otimes 
     \overline v$, letting $x= \xi\otimes a(\xi,
     \mbox{\textvisiblespace})\in Q$, we have by definition of
     $f_{[s]}$ 
     \[
       f_{[s]}\bigl((\Phi_{h_s}\circ\Theta')(v\otimes\xi,
      v\otimes\xi)\bigr) = T_s(v\,x\otimes\overline v) =
      \Trd_Q\bigl(s(v,v\,x)\bigr).
    \]
   Since $s(v,v\,x) = s(v,v)x=(s(v,v)\xi)\otimes a(\xi,
    \mbox{\textvisiblespace})$, the equation
    $\Trd_Q\bigl(s(v,v\,x)\bigr) = a(\xi, s(v,v)\xi)$ follows
    from~\eqref{eq:Trdsplit}. 
  \end{proof}

\begin{remk}
  Proposition~\ref{prop:quadpairs} yields another proof that $[s]$ is
  isotropic if and only if $q_{a*s}$ is isotropic, because $q_{a*s}$
  is isotropic if and only if $(\sigma_{b_{a*s}},f_{q_{a*s}})$ is
  isotropic, and Proposition~\ref{prop:isot} shows that $[s]$ is
  isotropic if and only if $(\sigma_{h_s}, f_{[s]})$ is isotropic.  
\end{remk}
\section{Generic splitting of quaternion algebras and the degree map}
\label{sec:gen}

Throughout this section, $Q$ is a quaternion division algebra
over an arbitrary field $k$. We fix a separable quadratic extension
$K$ of $k$ in $Q$ and write $Q$ as a crossed product
\[
  Q= (K,\iota,b) = K\oplus Kj
\]
where $\iota$ is the nontrivial $k$-automorphism of $K$ and $j\in Q$
satisfies
\[
  j^2=b\in k^\times \qquad\text{and}\qquad jx=\iota(x)j \quad\text{for
    $x\in K$.}
\] 

Let $t$ be an indeterminate over $K$. Extend $\iota$ to an
automorphism of the rational function field $K(t)$ by setting
\[
  \iota(t) = bt^{-1}.
\]
The function field $\FQ$ of the Severi--Brauer conic of $Q$ is a
generic splitting field of $Q$; recall from \cite[Remark~A.8]{MT} that
it can be viewed as the subfield of $K(t)$ fixed under $\iota$:
\[
  \FQ=K(t)^\iota.
\]
Every $k$-base $(1,\ell)$ of $K$ is also an $\FQ$-base of $K(t)$,
since every $f\in K(t)$ can be written as
\begin{equation}
  \label{eq:base}\
  f=f_0+\ell f_1 \quad\text{with}\quad
  f_0=\frac{\iota(\ell)f-\ell\iota(f)}{\iota(\ell)-\ell}\in \FQ
  \quad\text{and}\quad
  f_1=\frac{\iota(f)-f}{\iota(\ell)-\ell}\in\FQ.
\end{equation}
Therefore, $K\otimes_k\FQ=K(t)$, and the inclusion $\FQ\subset K(t)$
yields an identification
\[
  Q\otimes_k\FQ = K(t)\oplus K(t)j = (K(t),\iota, b).
\]
We write $Q_\FQ$ for the $\FQ$-algebra $Q\otimes_k\FQ$. 
It is split, as expected, since $j^2=b=\iota(t)t$. Note that $t$ does not belong to $F$, hence it is not a central element in $Q_F$. 
Consider
\[
  \varepsilon = 1+jt^{-1}=1+b^{-1}tj\in Q_{\FQ}
  \qquad\text{and}\qquad
  I=Q_{\FQ}\cdot\varepsilon\subset Q_{\FQ}.
\]
The following equations will be used repeatedly:
\begin{equation}
  \label{eq:fond}
  j\varepsilon = bt^{-1}+j = bt^{-1}\varepsilon=\iota(t)\varepsilon.
\end{equation}

\begin{prop}
  \label{prop:I}
  The left ideal $I$ in $Q_\FQ$ satisfies $I=K(t)\varepsilon=
  K\otimes_k\FQ\varepsilon$, and the expression of an element in $I$
  as $f\varepsilon$ with $f\in K(t)$ is unique. Moreover, identifying
  each $x\in Q_\FQ$ with multiplication on the left by $x$ yields
  an identification $Q_\FQ=\End_\FQ(I)$. 
\end{prop}

\begin{proof}
  The equations~\eqref{eq:fond} show that $K(t)j\varepsilon =
  K(t)\varepsilon$, hence 
  $I=K(t)\varepsilon$. Moreover, $K(t)=K\otimes_k\FQ$, hence
  $K(t)\varepsilon= K\otimes_k\FQ\varepsilon$. Multiplication by
  $\varepsilon$ is an injective map $K(t)\to I$
  since $K(t)$ is a field, hence every element in $I$ has a unique
  expression as $f\varepsilon$ with $f\in K(t)$, and it follows that
  $\dim_\FQ I=2$. The left $Q_\FQ$-module structure on $I$ yields a
  map $Q_\FQ\to\End_\FQ(I)$, which is injective because $Q_\FQ$ is a
  simple ring, hence surjective by dimension count.
\end{proof}

In preparation for the Morita equivalence between generalized
quadratic forms over $Q_F$ and ordinary quadratic forms over $F$ in
the next section, we now define a nonzero alternating form on $I$. 
We use the unique expression of elements of $I$ as $f\varepsilon$ with
$f\in K(t)$, and the choice of an element $u\in K^\times$ such that
$\iota(u)=-u$. Note that $u$ is
uniquely determined up to a factor in $k^\times$. If $\charac k=2$ we may take $u=1$.

\begin{prop}
  \label{prop:defa}
  The map $a\colon I\times I\to \FQ$ defined by
  \[
    a(f\varepsilon, g\varepsilon) = u(\iota(f)g-f\iota(g))
    \qquad\text{for $f$, $g\in K(t)$}
  \]
  is a nonzero alternating bilinear form. For $x\in Q$ and $\xi\in I$,
  the equation 
  $a(\xi, x\xi)=0$ holds if and only if $x\in k$ or $\xi=0$.
\end{prop}

\begin{proof}
  It is clear from the definition that $a$ is an alternating bilinear
  form. If $(1,\ell)$ is a $k$-base of $K$, then $a(\varepsilon,
  \ell\varepsilon) = u(\ell-\iota(\ell))\neq0$, hence $a$ is
  nonzero.

  If $\xi=0$ or $x\in k$, then $a(\xi,x\xi)=0$ because $a$ is
  $k$-linear and alternating.
  Now, suppose $\xi\in I$ is nonzero, and write $\xi=f\varepsilon$
  with $f\in K(t)^\times$. For $x=x_0+x_1j$, with $x_0,x_1\in K$,  
  using~\eqref{eq:fond} we obtain $x\xi =
  (x_0f+x_1\iota(ft))\varepsilon$, and the definition of $a$
  yields 
  \begin{align*}
    a(\xi, x\xi) &=
    u\bigl(\iota(f)(x_0f+x_1\iota(ft)) - 
    f\iota(x_0f+x_1\iota(ft))\bigr)
    \\ &=
    uf\iota(f)(x_0-\iota(x_0)
    +x_1\iota(ft)f^{-1}-\iota(x_1)ft\iota(f)^{-1}). 
  \end{align*}
  Therefore, letting $g=\iota(ft)f^{-1}\in K(t)^\times$, the equation
  $a(\xi, x\xi) =0$ implies 
  \[
    \bigl(x_0-\iota(x_0)\bigr) +x_1g-\iota(x_1g)=0.
  \]
  Note that $g\iota(g)=b$, hence $g\notin K$ because $Q$ is a division
  algebra. It follows that there exists a valuation $v$ on $K(t)$,
  trivial on $K$, such that $v(g)>0$. Then $v\bigl(\iota(g)\bigr)<0$
  since $g\iota(g)=b$ and $v(b)=0$, hence
  the three terms on the left side of the last displayed equation
  have different valuations if they are nonzero. Their sum vanishes
  only if they are all zero, which means that $x_0\in k$ and $x_1=0$.
\end{proof}

The main tool for the proof of the main theorem is a degree map on a
subring of $Q_{\FQ}$, which we now introduce. Consider the subring
$\fq$ of $K[t,t^{-1}]$ fixed under $\iota$,
\[
  \fq=K[t,t^{-1}]^\iota\subset \FQ,
  \qquad\text{and}\qquad
  \calQ=Q\otimes_k\fq\subset Q_\FQ.
\]

\begin{prop}
  \label{prop:defcalQ}
  The field of fractions of $\fq$ is $\FQ$, and
  \[
    K\otimes_k\fq = K[t,t^{-1}],
    \qquad
    \calQ= K[t,t^{-1}]\oplus K[t,t^{-1}]j.
  \]
  Every element in $\calQ$ has a unique expression of the form
  $\sum_{z\in\Z}x_zt^z$ with $x_z\in Q$ for all $z\in \Z$ and
  $x_z\neq0$ for only a finite number of $z\in\Z$.
\end{prop}

\begin{proof}
  Since $K(t)$ is the field of fractions of $K[t,t^{-1}]$,
every element in $\FQ$ can be written as $gh^{-1}=
g\iota(h)\cdot(h\iota(h))^{-1}$ with $g$, $h\in
K[t,t^{-1}]$ and $g\iota(h)$, $h\iota(h)\in\fq$. Hence $\FQ$
is the field of fractions of $\fq$. Using~\eqref{eq:base}, we
see that $K\otimes_k\fq=K[t,t^{-1}]$, so $\calQ=K[t,t^{-1}]\oplus
K[t,t^{-1}]j$. Since $tj=bjt^{-1}$, we have $K[t,t^{-1}]j =
jK[t,t^{-1}]$, and it follows that every element in $\calQ$ can be
uniquely written in the stated form.
\end{proof}

The proposition readily shows that $\varepsilon\in\calQ$. We may
therefore consider the left ideal in $\calQ$ generated by
$\varepsilon$,
\[
  \calI = \calQ\cdot\varepsilon\subset\calQ.
\]
Using Proposition~\ref{prop:defcalQ} and \eqref{eq:fond}, which shows
that $j\varepsilon\in K[t,t^{-1}]\varepsilon$, we see that
\[
  \calI = K[t,t^{-1}]\varepsilon = K\otimes_k\fq\varepsilon.
\]

A degree map on $\calQ$ is defined from the expression of each
$\mathfrak{X}\in \calQ$ as
\[
  \mathfrak{X}= \sum_{z\in\Z}x_zt^z
\]
with $x_z\in Q$ for all $z\in\Z$, and $x_z\neq0$ for only a finite
number of $z\in\Z$. If $\mathfrak{X}\neq0$, we define its degree
in~$\mathbb{N}$  
by comparing the absolute values of the exponents:
\[
  \deg(\mathfrak{X}) = \max\{\lvert z\rvert \mid
  x_z\neq0\}.
\]
We also let $\deg(0)=-\infty$, with the usual convention that
$-\infty<z$ for all $z\in\Z$.

\begin{remk}
  The restriction of the degree to $K[t,t^{-1}]$ is related to the
  $t$-adic and the $t^{-1}$-adic valuations on $K(t)$ by
  \[
    \deg(f) = - \min\{v_t(f),
    v_{t^{-1}}(f)\} = -\min\{v_t(f),\,v_t\bigl(\iota(f)\bigr)\}
    \qquad\text{for $f\in K[t,t^{-1}]$.}
  \]
  The valuations $v_t$ and $v_{t^{-1}}$ restrict to the same valuation
  $v_\infty$ on $\FQ$; they are the extensions of $v_\infty$ to
  $K(t)$, see \cite[Remark~A.8]{MT}. (This result is not used in the
  sequel.) 
\end{remk}

The degree map induces a filtration of $\calI$, for which the next proposition provides two equivalent descriptions. 

\begin{prop}
  \label{prop:degI}
  For $d\in\mathbb{N}$, let $\calI_{\leq d} =
  \{\xi\in\calI\mid 
  \deg(\xi)\leq d\}$. Then
  \[
    \calI_{\leq0}=\{0\} \qquad\text{and}\qquad
    \calI_{\leq d} = \Bigl\{\sum_{i=0}^{d-1}x_it^i\varepsilon \mid
    x_i\in Q\Bigr\}
    =\Bigl\{\sum_{z=-d}^{d-1} y_zt^z\varepsilon \mid
    y_z\in K\Bigr\}\qquad\text{for $d\geq1$.}
  \]
\end{prop}

\begin{proof}
  Since $\calI=K[t,t^{-1}]\varepsilon$, every element in $\calI$ has
  the 
  form $\sum_{z=-m}^ny_zt^z\varepsilon$ for some $y_{-m}$, \ldots,
  $y_n\in K$ and $m$, $n\geq0$.
  If $y_zt^z\varepsilon\neq0$, then expanding $\varepsilon$ and using
  $t^zj = j\iota(t^z) = jb^zt^{-z}$, 
  we find 
  \[
    y_zt^z\varepsilon = y_zt^z +y_zt^zjt^{-1} = y_zt^z +
    b^zy_zjt^{-z-1},
  \]
  hence
  \[
    \deg(y_zt^z\varepsilon) = \max\{\lvert z\rvert,\; \lvert
    -z-1\rvert\}=
    \begin{cases}
      z+1&\text{if $z\geq0$,}\\
      -z&\text{if $z<0$.}
    \end{cases}
  \]
  This shows that $\sum_{z=-d}^{d-1}y_zt^z\varepsilon\in\calI_{\leq
    d}$ for $y_{-d}$, \ldots, $y_{d-1}\in K$.
  
  Now, suppose $y_{-m}$, \ldots, $y_n\in K$ for some $m$, $n\geq0$,
  and
  $
    \deg\bigl(\sum_{z=-m}^ny_zt^z\varepsilon\bigr) \leq d.
  $
  Since the degree of a sum of terms of different degrees is the
  maximum of the degrees of the terms, we must have
  $
   (y_nt^n+y_{-n-1}t^{-n-1})\varepsilon = 0$ if $n\geq d$, hence $y_n=y_{-n-1}=0$ because by
    Proposition~\ref{prop:I} the expression of an element in $I$ as
    $f\varepsilon$ with $f\in K(t)$ is unique. Therefore, $n\leq d-1$
    if $y_n\neq0$ and $m\leq d$ 
  if $y_{-m}\neq0$. Hence
  \[
    \calI_{\leq0}=\{0\} \qquad\text{and}\qquad
    \calI_{\leq d} = \Bigl\{\sum_{z=-d}^{d-1} y_zt^z\varepsilon \mid
    y_z\in K\Bigr\}.
  \]
  To complete the proof, observe that $jt^{-z-1}\varepsilon =
  b^{-z-1}t^{z+1}j\varepsilon= b^{-z}t^z\varepsilon$ for all
  $z\in\Z$. Therefore, for $y_z\in K$,
  \[
    y_zt^z\varepsilon = b^zy_zjt^{-z-1}\varepsilon.
  \]
  Using the substitution $i=-z-1$, it follows that $\sum_{z=-d}^{-1}
  y_zt^z\varepsilon = 
  \sum_{i=0}^{d-1} b^{-i-1}y_{-i-1}jt^i\varepsilon$, hence
  \begin{equation}
    \label{eq:degI}
    \sum_{z=-d}^{d-1} y_zt^z\varepsilon = \sum_{i=0}^{d-1}
    (y_i+b^{-i-1}y_{-i-1}j)t^i\varepsilon. 
  \end{equation}
  On the other hand, for $x_0$, \ldots, $x_{d-1}\in Q$ we may find
  $y_{-d}$, \ldots, $y_{d-1}\in K$ such that $x_i=y_i+b^{-i-1}y_{-i-1}j$
  for $i=0$, \ldots, $d-1$; then~\eqref{eq:degI} yields
  \[
    \sum_{i=0}^{d-1}x_it^i\varepsilon = \sum_{z=-d}^{d-1}
    y_zt^z\varepsilon.
  \]
  Therefore, $
    \bigl\{ \sum_{z=-d}^{d-1} y_zt^z\varepsilon\mid y_z\in K\bigr\} =
    \bigl\{\sum_{i=0}^{d-1} x_it^i\varepsilon \mid x_i\in Q\bigr\}$
    and the proof is complete.
  \end{proof}

The following degree computation related to the alternating form $a$
of Proposition~\ref{prop:defa} is needed in Lemma~\ref{lem:deg} below.

\begin{lemma}
  \label{lem:degcomp}
  For $\xi$, $\eta\in\calI$, we have $a(\xi,\eta)\in\fq$. Moreover,
  for $d$, $e\in\mathbb{N}$ and $x\in Q$,
  \[
    \deg a(t^d\varepsilon, xt^e\varepsilon) \leq d+e+1.
  \]
  If $\deg a(t^d\varepsilon, xt^e\varepsilon)< d+e+1$, then $x\in K$.
\end{lemma}

\begin{proof}
  Since $\calI= K\otimes_k\fq\varepsilon$, it is clear from the
  definition that $a(\xi,\eta)\in\fq$ for $\xi$, $\eta\in\calI$.

  Let $x=x_0+x_1j\in Q$ with $x_0$, $x_1\in K$. Using~\eqref{eq:fond},
  we obtain
  \[
    xt^e\varepsilon = (x_0t^e +
    x_1b^{e+1}t^{-e-1})\varepsilon,
  \]
  hence the definition of $a$ yields
  \begin{align*}
    a(t^d\varepsilon, xt^e\varepsilon)
    & = u\bigl(\iota(t^d)(x_0t^e +
    x_1b^{e+1}t^{-e-1})-t^d\iota(x_0t^e +
    x_1b^{e+1}t^{-e-1})\bigr)
    \\
    & = u\bigl(x_0b^dt^{e-d} - \iota(x_0)b^et^{d-e} +
      x_1b^{d+e+1}t^{-d-e-1} - \iota(x_1)t^{d+e+1}\bigr).
  \end{align*}
  Since $d$, $e\geq0$, this last expression readily shows that $\deg
  a(t^d\varepsilon, xt^e\varepsilon)\leq d+e+1$. Moreover, the
  inequation is strict if and only if $x_1=0$, which amounts to $x\in
  K$. 
\end{proof}

\section{Proof of the theorem}
\label{sec:proof}

The argument is given at the end of this section. The core of the
proof lies in Lemma~\ref{lem:key}, which is a crucial tool to reduce
the degree of a given isotropic vector. We start with some technical
preparation, which mostly consists of some degree computations.

Throughout this section we fix a generalized quadratic space
$(V,[s])$ over $Q$ and an element $u\in K^\times$ such that
$\iota(u)=-u$, which yields a nonzero alternating bilinear form
$a\colon I\times I\to F$ as in Proposition~\ref{prop:defa}. After
scalar extension to $F$, the form $a$ is used to define a quadratic
form $q_{a*s}$ on the $F$-vector space $V\otimes_QF$ by the Morita
equivalence of \S\,\ref{subsec:Morita}. We write simply $q$ for
$q_{a*s}$ and $b_q$ for its polar form:
\[
  q\colon V\otimes_QI\to F, \qquad q(v\otimes\xi) = a(\xi,s(v,v)\xi)
  \qquad\text{for $v\in V$ and $\xi\in I$,}
\]
and, for $v$, $w\in V$ and $\xi$, $\eta\in I$,
\[
  b_q(v\otimes\xi, w\otimes\eta) = q(v\otimes\xi+w\otimes\eta) -
  q(v\otimes\xi) - q(w\otimes\eta) = a(\xi,h_s(v,w)\eta).
\]
Proposition~\ref{prop:nonsing} shows that $q$ is isotropic if and only
if the generalized quadratic form on the $Q_\FQ$-module
$V\otimes_k\FQ$ 
obtained from $[s]$ by scalar extension from $k$ to $F$ is
isotropic. Our goal is to show that $[s]$ is isotropic if $q$ is
isotropic.
\medbreak

Since $\FQ$ is the field of fractions of $\fq$ by
Proposition~\ref{prop:defcalQ}, it follows that $I=\calI\otimes_\fq
\FQ$, hence for every $v\in V\otimes_QI$ there exists $\lambda\in
\FQ^\times$ such that $v\lambda\in V\otimes_Q\calI$. As
$q(v\lambda)=0$ if and only if $q(v)=0$, it suffices to consider
isotropic vectors in $V\otimes_Q\calI$. Note that the restrictions of
$q$ and $b_q$ to $V\otimes_Q\calI$ take values in $\fq$, because
$a(\calI,\calI)\subset\fq$: see Lemma~\ref{lem:degcomp}. The degree
filtration on $\calI$ yields a filtration on $V\otimes_Q\calI$: each
vector in $V\otimes_Q\calI$ lies in $V\otimes_Q\calI_{\leq d}$ for
some $d\in\mathbb{N}$.

\begin{lemma}
  \label{lem:deg}
  For $d\in\mathbb{N}$,
  \[
    V\otimes_Q\calI_{\leq d+1} = \Bigl\{
    \sum_{i=0}^d v_i\otimes t^i\varepsilon\mid v_0,\,\ldots,\,v_d\in
    V\Bigr\}.
  \]
  Let $d$, $e\in\mathbb{N}$. If
  \[
    v=(v_{d}\otimes t^{d}\varepsilon) + v'\in V\otimes_Q\calI_{\leq d+1}
    \qquad\text{and}\qquad
    w = (w_{e}\otimes t^{e}\varepsilon) + w'\in V\otimes_Q\calI_{\leq e+1}
  \]
  with $v_{d}$, $w_{e}\in V$ and $v'\in V\otimes_Q\calI_{\leq
    d}$, $w'\in V\otimes_Q\calI_{\leq e}$, then
  \[
    \deg b_q(v,w) \leq d+e+1 \qquad\text{and}\qquad
    \deg q(v) \leq 2d+1.
  \]
  Moreover, if $\deg b_q(v,w)< d+e+1$, then $h_s(v_{d},
  w_{e})\in K$. If $\deg q(v) <2d+1$, then
  $s(v_{d},v_{d})\in K$.
\end{lemma}

\begin{proof}
  If $(e_j)_{j=1}^n$ is a $Q$-base of $V$, vectors in 
  $V\otimes_Q\calI_{\leq d+1}$ can be written as $\sum_{j=1}^n
  e_j\otimes\xi_j$ with $\xi_1$, \ldots, $\xi_n\in\calI_{\leq d+1}$.
  Proposition~\ref{prop:degI} shows that every $\xi_j\in\calI_{\leq
    d+1}$ has an expression $\xi_j=\sum_{i=0}^{d}
  x_{ji}t^i\varepsilon$ for some $x_{ji}\in Q$. Then
  \[
    \sum_{j=1}^n e_j\otimes\xi_j = \sum_{i=0}^{d} \bigl(\sum_{j=1}^n
    e_jx_{ji}\bigr) \otimes t^i\varepsilon,
  \]
  hence every vector in $V\otimes_Q\calI_{\leq d+1}$ can be
  written in the form
  $
  \sum_{i=0}^{d}v_i\otimes t^i\varepsilon
  $
  for some $v_0$, \ldots, $v_{d}\in V$.

  Now, let $v\in V\otimes_Q\calI_{\leq d+1}$ and $w\in
  V\otimes_Q\calI_{\leq e+1}$ be as in the statement.
  We first prove the part concerning $b_q(v,w)$, by induction on
  $d+e$. For this, we expand
  \begin{equation}
    \label{eq:exp}
    b_q(v,w) = b_q(v_{d}\otimes t^{d}\varepsilon,
    w_{e}\otimes t^{e}\varepsilon) + b_q(v_{d}\otimes
    t^{d}\varepsilon, w') + b_q(v',w_{e}\otimes
    t^{e}\varepsilon) +b_q(v',w').
  \end{equation}
  The first term on the right side is
  \[
    b_q(v_{d}\otimes t^{d}\varepsilon,
    w_{e}\otimes t^{e}\varepsilon) = a(t^{d}\varepsilon,
    h_s(v_{d},w_{e})t^{e}\varepsilon). 
  \]
  Lemma~\ref{lem:degcomp} shows that
  the degree of this term is at most $d+e+1$, and that $h_s(v_{d},
  w_{e})\in K$ 
  when its degree is strictly less than $d+e+1$. The last three terms
  on the right side 
  of~\eqref{eq:exp} vanish if $d=e=0$, for then $v'=w'=0$, and if
  $d+e\geq1$ their degree 
  is at most $d+e$ by the induction hypothesis. The part of the
  lemma related to $b_q(v,w)$ is thus proved.

  The other part is proved similarly, by induction on $d$: we have
  \begin{equation}
    \label{eq:exp2}
    q(v) = q(v_{d}\otimes t^{d}\varepsilon) +
    b_q(v_{d}\otimes t^{d}\varepsilon,v') + q(v').
  \end{equation}
  The first term on the right is
  \[
    q(v_{d}\otimes t^{d}\varepsilon) = a(t^{d}\varepsilon,
    s(v_{d},v_{d})t^{d}\varepsilon). 
  \]
  Its degree is at most~$2d+1$ by
  Lemma~\ref{lem:degcomp}; and if it is strictly less than
  $2d+1$, then $s(v_{d},v_{d})\in K$. The other terms on the right
  side of~\eqref{eq:exp2} vanish if 
  $d=0$. If $d\geq1$, the degree of the second term is at most $2d$
  by the first part of the proof, and the degree of the third term is
  at most $2d-1$ by the induction hypothesis. Therefore,
  $\deg q(v) \leq 2d+1$.
\end{proof}

\begin{corol}
  \label{corol:deg}
  Let $v=(v_{d}\otimes t^{d}\varepsilon)+v'\in
  V\otimes_Q\calI_{\leq d+1}$, with $v_{d}\in V$ and $v'\in
  V\otimes_Q\calI_{\leq d}$, be an isotropic vector of $q$.
  If
  $s(v_{d},v_{d})\notin k$, then $h_s(v_{d},v_{d})\neq0$.
\end{corol}

\begin{proof}
  Suppose $v_{d}\neq0$ and $h_s(v_{d},v_{d})=0$. We have to show
  that $s(v_{d},v_{d})\in k$. By definition, $h_s(v_{d},v_{d})
  = s(v_{d},v_{d})-\overline{s(v_{d},v_{d})}$, hence
  $s(v_{d},v_{d})=\overline{s(v_{d},v_{d})}$. If $s(v_d,v_d)\notin k$,
  then $s(v_d,v_d)\notin K$ since $k=\{x\in K\mid \overline
  x=x\}$. Lemma~\ref{lem:deg} then yields 
  \[
    \deg q(v_{d}\otimes t^{d}\varepsilon) = 2d+1.
  \]
  Now, from $q(v)=0$ it follows that
  \[
    q(v_{d}\otimes t^{d}\varepsilon) = -b_q(v_{d}\otimes
    t^{d}\varepsilon, v') - q(v'),
  \]
  but Lemma~\ref{lem:deg} yields $\deg b_q(v_{d}\otimes
  t^{d}\varepsilon,v')\leq 2d$ and $\deg q(v')
  \leq 2d-1$, so this equation is impossible.
\end{proof}

The degree reduction argument in the proof of the theorem
relies on the following key result:

\begin{lemma}
  \label{lem:key}
  Let $v\in V$ and $\xi\in\calI$ be such that $\deg(\xi)=d+1\geq2$. If
  $\deg q(v\otimes\xi)\leq2d-1$, then
  $q(v\otimes\xi) = q(v\otimes\eta)$ for some
  $\eta\in\calI_{\leq d}$.
\end{lemma}

\begin{proof}
  If $q(v\otimes\xi)=0$, the lemma holds with $\eta=0$. For the rest
  of the proof, we assume $q(v\otimes\xi)\neq0$.
  By Proposition~\ref{prop:degI} we may write
  $\xi=(x_{d}t^{d}+\cdots+x_0)\varepsilon$ with $x_{d}$,
  \ldots, $x_0\in Q$ and $x_{d}\neq 0$. Substituting $vx_d$ for $v$,
  we may (and will) assume $x_d=1$. Lemma~\ref{lem:deg} then yields
  $s(v,v)\in K$. If $s(v,v)\in k$, then
  \[
    q(v\otimes\xi) = a(\xi,s(v,v)\xi) = s(v,v) a(\xi,\xi)=0.
  \]
  This case is excluded since
  $q(v\otimes\xi)\neq0$, hence $s(v,v)\in K\setminus k$.

  By Proposition~\ref{prop:degI} we
  have
  \[
    (x_{d-2}t^{d-2}+\cdots+x_0)\varepsilon = (y_{d-2}t^{d-2}+\cdots +
    y_{-d+1}t^{-d+1})\varepsilon
  \]
  for some $y_{d-2}$, \ldots, $y_{-d+1}\in K$. We prove below:
  \medbreak

  \noindent\textit{Claim:} $x_{d-1}$ lies in $K$.
  \medbreak

  Assuming the claim, we complete the proof as follows: we write
  \[
    \xi=z\varepsilon \qquad\text{with}\qquad
    z=t^d+x_{d-1}t^{d-1}+y_{d-2}t^{d-2}+\cdots+y_{-d+1}t^{-d+1} \in
    K[t,t^{-1}].
  \]
  Since $s(v,v)\in K$, we get, by definition of $q$ and $a$,
  \begin{equation}
    \label{eq:key}
    q(v\otimes\xi) = a(\xi,s(v,v)\xi) = u\bigl(\iota(z)s(v,v)z -
    z\iota(s(v,v)z)\bigr) =
    uz\iota(z)\bigl(s(v,v)-\iota(s(v,v))\bigr).
  \end{equation}
  The rightmost expression does not change when we change $z$ into
  $\iota(z)$, 
  hence $q(v\otimes\xi) = q(v\otimes\iota(z)\varepsilon)$. 
  The lemma will follow because
  \[
    \iota(z)\varepsilon=(b^dt^{-d}+b^{d-1}\iota(x_{d-1})t^{1-d} +
    b^{d-2}\iota(y_{d-2}) t^{2-d} + \cdots +
    b^{-d+1}\iota(y_{-d+1})t^{d-1})\varepsilon \in \calI_{\leq d}.
  \]
  \medbreak

  To prove the claim,
  let $\xi_0=(x_{d-2}t^{d-2}+\cdots+x_0)\varepsilon\in \calI_{\leq
    d-1}$, so $\xi=(t^d+x_{d-1}t^{d-1})\varepsilon+\xi_0$ and
  \[
    q(v\otimes\xi) = q(v\otimes(t^d+x_{d-1}t^{d-1})\varepsilon) +
    b_q(v\otimes(t^d+x_{d-1}t^{d-1})\varepsilon, v\otimes\xi_0) +
    q(v\otimes\xi_0).
  \]
  Lemma~\ref{lem:deg} shows that the middle
  term on the right side has degree at most $2d-1$ and the rightmost
  term has degree at most $2d-3$. Therefore, the hypothesis that $\deg
  q(v\otimes\xi)\leq 2d-1$ implies
  \begin{equation}
    \label{eq:key2}
    \deg q(v\otimes(t^d+x_{d-1}t^{d-1})\varepsilon) \leq
    2d-1.
  \end{equation}
  Now, we have
  \[
    q(v\otimes(t^d+x_{d-1}t^{d-1})\varepsilon)=q(v\otimes
    t^d\varepsilon) + b_q(v\otimes t^d\varepsilon, v\otimes
    x_{d-1}t^{d-1}\varepsilon) + q(v\otimes x_{d-1}t^{d-1}\varepsilon)
  \]
  and Lemma~\ref{lem:deg} shows that the rightmost term has degree at
  most $2d-1$. Moreover, since $s(v,v)\in K\setminus
  k$, the same computation as for~\eqref{eq:key} yields
  \[
    q(v\otimes t^d\varepsilon) = ub^d\bigl(s(v,v)-\iota(s(v,v))\bigr)
    \in K^\times,
  \]
  hence $\deg q(v\otimes t^d\varepsilon) = 0$. Therefore,
  \eqref{eq:key2} implies
  \[
    \deg b_q(v\otimes t^d\varepsilon,v\otimes
    x_{d-1}t^{d-1}\varepsilon)\leq 2d-1,
    \qquad\text{i.e.,}\quad
    \deg a(t^d\varepsilon, h_s(v,vx_{d-1})t^{d-1}\varepsilon) \leq
    2d-1, 
  \]
  hence $h_s(v,\,vx_{d-1})\in K$ by
  Lemma~\ref{lem:deg}. But
  \[
    h_s(v,\,vx_{d-1}) = h_s(v,\,v)x_{d-1} = \bigl(s(v,v) -
    \overline{s(v,v)}\bigr) x_{d-1}
  \]
  and $s(v,v)-\overline{s(v,v)}\in K^\times$ because $s(v,v)\in
  K\setminus k$, hence $x_{d-1}\in K$. This completes the proof.
\end{proof}

\begin{proof}[Proof of the theorem]
  Suppose $q$ is isotropic. As observed at the beginning of this
  section, by clearing denominators we may find an isotropic vector in
  $V\otimes_Q\calI$. We will
  show that the minimal $d\in\mathbb{N}$ such that
  $V\otimes_Q\calI_{\leq d+1}$ 
  contains a $q$-isotropic vector is~$0$; it will quickly follow
  that $s(v,v)\in k$ for some nonzero $v\in V$, which means that $[s]$
  is isotropic and proves the theorem. 

  Suppose that $V\otimes_Q\calI_{\leq d+1}$ contains an isotropic
  vector 
  for some $d\geq1$. This vector has the form
  $\sum_{i=0}^{d}v_i\otimes t^i\varepsilon$ for some $v_0$, \ldots,
  $v_{d}\in V$. It is already in $V\otimes_Q\calI_{\leq
    d}$ if $v_{d}=0$, so we need to consider only the case where
  $v_{d}\neq0$. We may also assume $s(v_{d},v_d)\notin k$,
  otherwise we readily get the $q$-isotropic vector
  $v_{d}\otimes\varepsilon\in V\otimes_Q\calI_{\leq1}$. 
  Now, Corollary~\ref{corol:deg} shows that $h_s(v_{d},v_{d})\neq0$.
  We may then project each $v_i$ on $v_{d}$: letting
  \[
    x_i=h_s(v_{d},v_{d})^{-1}h(v_{d},v_i)
    \quad\text{and}\quad
    v'_i=v_i-v_{d}x_i
    \quad\text{for $i=0$, \ldots, $d-1$,}
  \]
  we get $h_s(v_{d},v'_i)=0$ for $i=0$, \ldots, $d-1$, which implies
  that $b_q(v_d\otimes\xi,v'_i\otimes\xi')=0$ for all $\xi$,
  $\xi'\in\calI$.  Rewrite the
  given $q$-isotropic vector
  \begin{equation}
    \label{eq:rewrite}
    \sum_{i=0}^{d} v_i\otimes t^i\varepsilon =
    v_{d}\otimes(t^{d}+x_{d-1}t^{d-1}+\cdots+x_0)\varepsilon +
    \Bigl(\sum_{i=0}^{d-1}v'_i\otimes t^i\varepsilon\Bigr).
  \end{equation}
  Since this vector is isotropic and $h_s(v_{d},v'_i)=0$ for all $i$,
  it follows that
  \[
    q(v_{d}\otimes(t^{d}+x_{d-1}t^{d-1}+ \cdots +
    x_0)\varepsilon) = -q\Bigl(\sum_{i=0}^{d-1} v'_i\otimes
    t^i\varepsilon\Bigr).
  \]
  Lemma~\ref{lem:deg} shows that the right side has degree
  $\leq2d-1$, hence Lemma~\ref{lem:key} yields $\eta\in\calI_{\leq
    d}$ such that
  \[
    q(v_{d}\otimes\eta) =
    q(v_{d}\otimes(t^{d}+x_{d-1}t^{d-1}+ \cdots + 
    x_0)\varepsilon).
  \]
  Then
  \[
    q\Bigl((v_d\otimes\eta) + \sum_{i=0}^{d-1} v'_i\otimes
    t^i\varepsilon\Bigl)=0
    \qquad\text{with}
    \qquad
    (v_d\otimes\eta) + \sum_{i=0}^{d-1} v'_i\otimes
    t^i\varepsilon\in V\otimes_Q\calI_{\leq d}.
  \]
  It remains to prove that this vector in nonzero, or equivalently, in view of~\eqref{eq:rewrite}, that 
  \[
    \sum_{i=0}^d v_i\otimes t^i\varepsilon \not= v_d\otimes\xi
    \qquad\text{where}\qquad
    \xi= (t^{d}+x_{d-1}t^{d-1}+ \cdots + 
    x_0)\varepsilon-\eta.
  \]
  To prove this, we use Proposition~\ref{prop:defa}. Since
  $\eta\in\calI_{\leq d}$, the element $\xi$ is nonzero. Moreover, we
  have $s(v_d,v_d)\not \in k$. Therefore,  
        $a(\xi, s(v_d,v_d)\xi)=q(v_d\otimes\xi)\not=0$, and it follows
        $v_d\otimes \xi$ is not equal to the isotropic vector
        $\sum_{i=0}^d v_i\otimes t^i\varepsilon$. 
 
  Therefore, $(v_d\otimes\eta) +
  \sum_{i=0}^{d-1} v'_i\otimes t^i\varepsilon$ is an isotropic vector
  in $V\otimes_Q\calI_{\leq d}$, 
  and we may repeat the argument if $d\geq2$ until we find a
  $q$-isotropic vector in $V\otimes_Q\calI_{\leq1}$.

  We have thus shown that $V\otimes_Q\calI_{\leq1}$ contains a
  $q$-isotropic vector. This vector has the form
  $v\otimes\varepsilon$ for some nonzero $v\in V$, and
  $a(\varepsilon,\, s(v,v)\varepsilon)=q(v\otimes\varepsilon) = 0$.
  By Proposition~\ref{prop:defa}, this implies that $s(v,v)\in k$,
  hence $[s]$ is isotropic. 
\end{proof}

\bibliographystyle{amsalpha}
\bibliography{QTconic}

\end{document}